\documentclass[11pt, reqno]{amsart}

\usepackage{amsmath, amsthm, amsfonts, amssymb, amsbsy, bigstrut, graphicx, enumerate,  upref, longtable, comment, booktabs, array, caption, subcaption}

\usepackage{pstricks}
\usepackage{times}

\usepackage[breaklinks]{hyperref}

\usepackage[numbers, sort&compress]{natbib} 

\newcommand{\tf}{\tilde{f}}

\newtheorem{thm}{Theorem}[section]
\newtheorem{lmm}[thm]{Lemma}
\newtheorem{cor}[thm]{Corollary}
\newtheorem{prop}[thm]{Proposition}

\theoremstyle{definition}

\newcommand{\ee}{\mathbb{E}}

\newcommand{\cp}{\mathcal{P}}
\newcommand{\cq}{\mathcal{Q}}

\newcommand{\pp}{\mathbb{P}}

\newcommand{\rr}{\mathbb{R}}

\newcommand{\var}{\mathrm{Var}}
\newcommand{\ve}{\varepsilon}

\newcommand{\zz}{\mathbb{Z}}

\newcommand{\tilh}{\tilde{h}}

\newcommand{\fpar}[2]{\frac{\partial #1}{\partial #2}}

\numberwithin{equation}{section}

\usepackage[utf8]{inputenc}
\usepackage{amsmath}
\usepackage{amsfonts}
\usepackage{bbm}
\usepackage{mathtools}
\usepackage{tikz}
\usetikzlibrary{decorations.pathreplacing}
\usepackage{amsthm}
\usepackage[english]{babel}
\usepackage{fancyhdr}
\usepackage{amssymb}
\usepackage{enumitem}
\usepackage{qtree}
\usepackage{mathrsfs}

\renewcommand{\tilde}{\widetilde}

\begin{document}

\title{Superconcentration in surface growth}
\author{Sourav Chatterjee}
\address{Departments of Mathematics and Statistics, Stanford University}
\email{souravc@stanford.edu}
\thanks{Research partially supported by NSF grant DMS-1855484}
\thanks{Data availability statement: Data sharing not applicable to this article as no datasets were generated or analyzed during the current study}

\begin{abstract}
Height functions of growing random surfaces are often conjectured to be superconcentrated, meaning that their  variances grow sublinearly in time. This article introduces a new concept --- called \emph{subroughness} ---  meaning  that there exist two distinct points such that the expected squared difference between the heights at these points grows sublinearly in time. The main result of the paper is that superconcentration is equivalent to subroughness in a class of growing random surfaces. The result is applied to establish superconcentration in a variant of the restricted solid-on-solid (RSOS) model and in a variant of the ballistic deposition model, and give new  proofs of superconcentration in directed last-passage percolation and directed polymers.
\end{abstract} 

\maketitle


\section{Introduction and results}\label{intro}
A $d$-dimensional growing random surface is represented as a height function $f:\zz_{\ge 0}\times \zz^d\to \rr$ evolving in time, where $f(t,x)$ denotes the height of the surface at location $x$ at time $t$. The simplest example is the {\it random deposition model}, where the height $f(t,x)$ at each $x$ grows as a random walk with i.i.d.~increments, independently of the heights at other locations. In this model, $\var(f(t,x))$ grows linearly in $t$.

This is not the case, however, for any nontrivial model of surface growth where the growth of the height at a point is influenced by the heights at neighboring points. For most such models, it is conjectured that $\var(f(t,x))$ grows sublinearly in $t$, often in a very specific manner depending on the model~\cite{familyvicsek91, kimkosterlitz89, kellingodor11, pagnaniparisi15}. These conjectures have been rigorously proved in only a handful of cases, mostly for $d=1$, where exact calculations are possible. For surveys of the vast literature on one-dimensional surface growth and some recent advances in higher dimensions, see \cite{quastel12, corwin16, toninelli18}. 

Beyond exactly solvable models, not much is known. Even just showing that $\var(f(t,x)) = o(t)$ as $t\to\infty$ seems to be a challenging problem in nontrivial models. This is sometimes called {\it superconcentration} of the height function~\cite{chatterjee14}. The only nontrivial surface growth models where superconcentration has been rigorously established are directed last-passage percolation and directed polymers~\cite{alexanderzygouras13, graham12,chatterjee14, chatterjee08}, building on technology developed in~\cite{benjaminietal03} for the related model of first-passage percolation. 

The main result of this article shows that in a certain class of surface growth models, $\var(f(t,x))$ grows sublinearly in $t$ if and only if there exist two distinct points $x$ and $y$ (usually neighbors) such that $\ee[(f(t,x)-f(t,y))^2]$ grows sublinearly in $t$. The latter phenomenon is named {\it subroughness} in this paper. 

The utility of the equivalence theorem is demonstrated by applying it to prove superconcentration in variants of two popular models of random surface growth: (a) the restricted solid-on-solid model, and (b) the ballistic deposition model. Additionally, the theory is applied to give new proofs of superconcentration in directed last-passage percolation and directed polymers. 

The main advantage of subroughness over superconcentration is that it may be easier to establish subroughness  because neighboring heights are often close to each other `by design'. We will see a clear instance of this in the RSOS model later. Moreover, the equivalence of subroughness and superconcentration is conceptually interesting, because it says that superconcentration in random surfaces is caused by the tendency of neighboring heights to remain close to each other.  

The rest of this section contains the details of the theory. Examples are presented in Section \ref{examples}. The remaining sections contain the proofs. 

 

\subsection{A class of surface growth models}\label{defsec}
Let $d$ be a positive integer. Let $e_1,\ldots,e_d$ be the standard basis vectors of $\rr^d$. Let $A$ denote the set $\{0, \pm e_1,\pm e_2,\ldots, \pm e_d\}$, consisting of the origin and its $2d$ nearest neighbors in $\zz^d$. Let $B := A\setminus \{0\}$. The sets $A$ and $B$ will be fixed throughout this paper.  Let $\phi:\rr^A\times \rr\to \rr$ be a function. Let $\mathbf{z} = \{z_{t,x} : t\in \zz_{>0}, x\in \zz^d\}$ be a collection of i.i.d.~random variables. We will say that the evolution of a $d$-dimensional growing random surface $f:\zz_{\ge 0} \times \zz^d \to \rr$ is driven by the function $\phi$ and the `noise field' $\mathbf{z}$ if for each $t\in \zz_{\ge0}$ and $x\in \zz^d$,
\begin{align}\label{kpzevolve2}
f(t+1,x) = \phi((f(t,x+a))_{a\in A}, z_{t+1,x}).
\end{align}
We will henceforth assume that $z_{t,x}$ are i.i.d.~standard Gaussian random variables. This will not be too restrictive, since the only assumption we will make about $\phi$, in relation to the noise field, is that $\phi$ is Lipschitz continuous in the second argument (see below). This allows the noise variables to be anything that can be expressed as a Lipschitz function of a Gaussian random variable (e.g., uniform). 

Equation~\eqref{kpzevolve2} generalizes the  mechanism considered in \cite{chatterjee21, chatterjeesouganidis21}, which is almost the same except that it does not involve randomness. 
We assume that $\phi$ has the following properties:
\begin{itemize}
\item {\it Equivariance under constant shifts.} For $u\in \rr^A$ and $c\in \rr$, let $u+c$ denote the vector obtained by adding $c$ to each coordinate of $u$. We assume that $\phi(u+c, z)=\phi(u,z)+c$ for each $u\in \rr^A$ and $z\in \rr$. 
\item {\it Monotonicity.} We assume that $\phi$ is monotone increasing in the first variable. That is, if $u$ dominates $v$ in each coordinate, then $\phi(u,z)\ge \phi(v,z)$ for any $z$. 
\item {\it Lipschitz continuity in the noise variable.} We assume that $\phi$ is Lipschitz in the second argument with a Lipschitz constant $L$. That is, for all $u\in \rr^A$ and $z,z'\in \rr$, $|\phi(u,z)-\phi(u,z')|\le L|z-z'|$. 
\end{itemize}
Examples that are {\it not} covered by the growth mechanism \eqref{kpzevolve2} include any model where every vertex has a Poisson clock attached to it, and an update happens whenever the clock rings. This is an important class of models, which are quite similar to~\eqref{kpzevolve2} but different enough so that the methods of this paper do not immediately generalize. It would be interesting and important to see if analogous methods can be developed for such models.

Incidentally, the assumptions of monotonicity and equivariance for a discrete evolution equation are widely used in the literature on approximation schemes for nonlinear partial differential equations, starting with~\cite{barlessouganidis91}. They are also two of the key assumptions in \cite{chatterjee21, chatterjeesouganidis21}.

\subsection{A general fluctuation bound}\label{genfluc}
Henceforth, let $f$ be a growing random surface with driving function $\phi$ and i.i.d.~standard Gaussian noise field $\mathbf{z}$, where $\phi$ has the monotonicity and equivariance properties, and is Lipschitz in the noise variable with Lipschitz constant $L$. Our first main result is the following theorem, which says that under the above conditions, $f(t,x)$ has fluctuations of order at most $L\sqrt{t}$.  For this result, $f(0,\cdot)$ can be any function on $\zz^d$. We will later assume that $f(0,\cdot)\equiv 0$. 
\begin{thm}\label{genvarthm}
For all $t\ge 1$ and  $x\in\zz^d$, $\var(f(t,x))\le L^2 t$. Moreover, for all $\theta\in \rr$,  
\[
\ee(e^{\theta(f(t,x) - \ee(f(t,x)))}) \le e^{L^2 t\theta^2/2},
\]
and for all $r\ge 0$, 
\[
\pp(|f(t,x)-\ee(f(t,x))| \ge r) \le 2e^{-r^2/2L^2t}.
\]  
\end{thm}
This theorem is proved in Section \ref{genvarthmproof}. The proof is based on the concentration of the Gaussian measure and a random walk representation of the derivatives of $f(t,x)$ with respect to the noise variables, derived in Section \ref{derivsec}.

\subsection{Equivalence of subroughness and superconcentration}\label{equivsec}
In this subsection, let us assume that $f(0,\cdot) \equiv 0$, in addition to the assumptions that the driving function $\phi$ is equivariant, monotone and Lipschitz in the noise variable with Lipschitz constant $L$, and that the noise variables $z_{t,x}$ are i.i.d.~standard Gaussian. We will say that the surface $f$ is {\it superconcentrated} if
\[
\lim_{t\to\infty} \frac{\var(f(t,x))}{t} = 0.
\]
Note that the term on the left does not depend on $x$ due to the assumption that $f(0,\cdot )\equiv 0$. We will say that the surface is {\it subrough} if there exist two distinct points $x,y\in \zz^d$ such that
\[
\lim_{t\to\infty}\frac{\ee[(f(t,x)-f(t,y))^2]}{t} =0.
\]
Lastly, we will say that the surface is {\it completely subrough} if the above equality holds for any two distinct points $x$ and $y$. The main result of this subsection (and of this paper) is the following.
\begin{thm}\label{equivthm}
For the surface $f$, superconcentration, subroughness and complete subroughness are equivalent.
\end{thm}
This result will be a consequence of a quantitative bound, which we now state. For each $t\ge 1$, define
\[
\alpha_t := \frac{\var(f(t,x))}{L^2t}. 
\]
Note that since $f(0,\cdot)\equiv 0$, the right side does not depend on $x$. Next, for any $b\in \zz^d$ and $t\ge 1$, define 
\[
\beta_{b,t} := \frac{\ee[(f(t,x)-f(t,x+b))^2]}{4L^2t}. 
\]
Again, note that the right side does not depend on $x$, but may depend on $b$. The surface is superconcentrated if and only if $\alpha_t\to 0$ as $t\to \infty$. On the other hand, the surface is subrough if and only if for some $b\ne 0$, $\beta_{b,t} \to 0$ as $t\to\infty$, and completely subrough if and only if this holds for any $b\ne 0$. The following theorem relates $\alpha_t$ and $\beta_{b,t}$ through a pair of inequalities, which immediately imply that these three conditions are equivalent, and hence establish Theorem \ref{equivthm}. The proof uses the ``$L^1$--$L^2$ bound''  of~\citet{talagrand94} (which is an extension of the idea of using hypercontractivity for improving variance bounds due to~\citet{kahnetal88}), and an averaging trick invented by \citet{benjaminietal03}. The main new ingredient in the argument is the random walk representation from Section \ref{derivsec}. 
\begin{thm}\label{equivthm2}
There is a universal constant $C$ such that for any $b\ne 0$ and $t\ge 1$, 
\[
\beta_{b,t} \le \alpha_t \le \frac{C}{|\log \beta_{b,t}|}. 
\]
\end{thm}
This theorem is proved in Section \ref{equivproof}. It would be interesting to understand if the upper bound is sharp under the given conditions, or if it can be improved.

\section{Examples}\label{examples}
This section contains the applications of the theory to the four examples mentioned in the introduction, namely, a variant of the RSOS model, a variant of ballistic deposition, directed last-passage percolation, and directed polymers.
\subsection{A variant of the RSOS model}\label{rsossec}
The restricted solid-on-solid (RSOS) model is a popular toy model of surface growth introduced by \citet{kimkosterlitz89} (not to be confused with an `eight vertex model' that goes by the same name~\cite{andrewsetal84}). There are many variants of this model, all built on one basic principle: The growing surface has to satisfy, at all times, that the differences between the heights at neighboring points are uniformly bounded by some given constant (usually $1$). 



We will work with the following variant in this subsection. Consider $\zz^d$ as a bipartite graph, splitting the set of vertices into `even' and `odd' vertices, depending on the parity of the sum of coordinate values. Alternately update the heights at even and odd vertices, choosing independently and uniformly among all values that maintain the constraint that  the differences between the heights at neighboring points are uniformly bounded by $1$. To be more explicit, the algorithm is as follows. Let $f(t,x)$ denote the height of the surface at time $t$ and location $x$. Then: 
\begin{itemize}
\item Start with $f(0,x)=0$ for all $x$.
\item If $t$ is even, then for each even vertex $x$, choose $f(t+1,x)$ uniformly from the interval 
\[
[\max_{b\in B} f(t,x+b) -1, \min_{b\in B} f(t,x+b)+1], 
\]
which is the set of all possible values that maintain the required constraint. (Recall that $B = \{\pm e_1,\ldots, \pm e_d\}$ is the set of nearest neighbors of the origin.) For each odd vertex, let $f(t+1,x)=f(t,x)$. 
\item If $t$ is odd, switch the update rules for odd and even vertices in the above step. 
\end{itemize}
With the above growth mechanism, it is easy to see inductively that the required constraint is maintained at all times. 

The growth of $\var(f(t,x))$ is one of the main unsolved questions about RSOS-type models. For $d=1$, it is believed that the variance grows like $t^{2/3}$, just like in any other model in the KPZ universality class~\cite{kimkosterlitz89}. For $d=2$, it was conjectured in \cite{kimkosterlitz89} that the variance grows like $t^{1/2}$, but this has been contradicted in some large-scale numerical studies in recent years~\cite{pagnaniparisi15, kellingodor11}. The following result  shows that in the variant described above, $\var(f(t,x))$ grows at most like $t/\log t$.  
\begin{thm}\label{rsosthm}
Let $f$ be the height function in the variant of the RSOS model defined above, in any dimension. There is a constant $C(d)$, depending only on the dimension $d$, such that for any $t\ge 2$ and $x\in \zz^d$, $\var(f(t,x))\le C(d)t/\log t$. 
\end{thm}
This result is proved in Section \ref{rsosproof}. The logarithmic correction comes from applying Theorem \ref{equivthm2}. Although the growth mechanism of the model does not exactly fit into the framework of this paper, this can be easily taken care of, as we will do in Section \ref{rsosproof}.



\subsection{A variant of ballistic deposition}\label{ballisticsec}
Ballistic deposition is a popular model of surface growth introduced by~\citet{vold59} and subsequently studied by many authors. One version of the model is as follows. There is, as usual, a height function $f(t,x)$, but now the time variable is continuous. There is an independent Poisson clock at each $x$. When the clock at $x$ rings, a brick of height $1$ drops on the surface at location $x$ `from infinity', as in a game of Tetris. As the brick descends, it can either attach itself to the surface at $x$, thereby increasing the height at $x$ by $1$, or it can get `stuck' to the side of one the neighboring columns if that happens before it reaches the surface. Thus, if the clock at $x$ rings at time $t$, then the height $f(t,x)$ instantly increases to 
\begin{align}\label{ball1}
\max\{f(t,x)+1, \max_{b\in B} f(t,x+b)\}. 
\end{align}
The physical literature on ballistic deposition is huge. For classical surveys, see~\cite{familyvicsek91,barabasistanley95}. Physicists say that this model is in the KPZ universality class, implying that the variance of $f(t,x)$ grows like $t^{2/3}$ when $d=1$~\cite{kimkosterlitz89}, and possibly like $t^{\alpha}$ for some $\alpha$ slightly less than $1/2$ when $d=2$~\cite{pagnaniparisi15}. On the mathematical side, the only results we know are the following:
\begin{itemize}
\item A strong law of large numbers for the height function was proved by \citet{seppalainen00}. 
\item A central limit theorem for the total height in a large region at a finite time $t$ was proved by \citet{penroseyukich02}. 
\item \citet{penrose08} proved that the variance of $f(t,x)$ grows at least like $\log t$ when $d=1$. 
\end{itemize}
In this section we will consider the following variant of ballistic deposition. Instead of bricks falling at random times, our model will update the heights at all sites simultaneously. To insert randomness, we will make the brick heights random. For definiteness, let us take the brick heights to be i.i.d.~Uniform$[0,1]$ random variables. In other words, the height function $f:\zz_{\ge 0} \times \zz^d \to \rr$ behaves as follows, in analogy with \eqref{ball1}. 
\begin{itemize}
\item We start with $f(0,x)=0$ for all $x$.
\item For each $t\ge 0$ and $x\in \zz^d$, we let
\[
f(t+1,x) = \max\{f(t,x)+v_{t+1,x}, \max_{b\in B} f(t,x+b)\},
\]
where $v_{t,x}$ are i.i.d.~Uniform$[0,1]$ random variables.
\end{itemize}
We will show that in this model $\var(f(t,x))\le C(d)t/\log t$, where $C(d)$ is a constant that depends only on $d$. To put this in the framework of equation \eqref{kpzevolve2}, we define $v_{t,x} = \Phi(z_{t,x})$ where $z_{t,x}$ are i.i.d.~standard Gaussian random variables and $\Phi$ is the standard Gaussian c.d.f., and then take
\begin{align}\label{ballphi}
\phi(u,z) = \max\{u_0 + \Phi(z), \max_{b\in B} u_b\}. 
\end{align}
Note that this $\phi$ is monotone, equivariant, and Lipschitz in the noise variable with Lipschitz constant bounded by $1/\sqrt{2\pi}$. 

We will, in fact, prove superconcentration of the surface for a broader class of driving functions that includes the above $\phi$ as a special case. This class of driving functions will be called `max type'. We will say that a driving function $\phi:\rr^A\times \rr\to \rr$ is of max type if it is monotone, equivariant, Lipschitz in the noise variable, and there are nonnegative constants $K_1$ and $K_2$ such that for all $u\in \rr^A$ and $z\in \rr$,
\begin{align}\label{maxtype}
|\phi(u,z)-\max_{a\in A} u_a|\le K_1+ K_2|z|.
\end{align}
Clearly, the $\phi$ displayed in \eqref{ballphi} is of max type. The following theorem shows that the surface generated by any model of max type (including, in particular, the variant of ballistic deposition introduced above) is superconcentrated. 
\begin{thm}\label{ballthm1}
Let $f$ be a growing random surface with $f(0,\cdot)\equiv 0$ and growing according to \eqref{kpzevolve2}, where $\phi$ is of max type, and the noise field is i.i.d.~standard Gaussian. Then there is a constant $C$ depending only on $\phi$ and $d$, such that:
\begin{enumerate}
\item[$(1)$] For any $t\ge 2$ and neighboring points $x,y\in \zz^d$, $\ee|f(t,x)-f(t,y)|\le Ct^{1/4}\sqrt{\log t}$. 
\item[$(2)$] For any $t\ge 2$ and neighboring points $x,y\in \zz^d$, $\ee[(f(t,x)-f(t,y))^2]\le Ct^{3/4}\log t$. 
\item[$(3)$] For any $t\ge 2$ and $x\in \zz^d$, $\var(f(t,x))\le Ct/\log t$. 
\end{enumerate}
\end{thm}
This result is proved in Section \ref{ballisticproof}. The proof of the first claim is by a new argument that may be of independent interest. The second claim follows by combining the first claim and Theorem \ref{genvarthm}. The third claim is proved using the second claim and Theorem \ref{equivthm2}.

\subsection{Point-to-plane last-passage percolation}\label{lppsec}
The model of $d$-dimensional point-to-plane directed last-passage percolation (LPP)~\cite{johansson00} fits into our framework, for any $d\ge 2$. Recall that this model is defined as follows. We start with a collection of i.i.d.~vertex weights $\{w_x\}_{x\in \zz^d}$, often called the `environment'. Let $|x|_1$ denote the $\ell^1$ norm of a vector $x\in \zz^d$. Let $O^+$ denote the positive orthant in $\zz^d$, that is, the set of vectors with nonnegative coordinates. Given an integer $t\ge 1$, let $\cq_t$ be the set of all lattice paths from the origin to the plane $\{(x_1,\ldots,x_d)\in O^+: |x|_1 = t\}$, which move in the `positive direction' at each step. In other words, an element $Q\in \cq_t$ is a sequence $(q_0,\ldots, q_t)\in (\zz^d)^t$ such that $q_0=0$, and for each $i\ge 1$, $q_i = q_{i-1}+e_j$ for some $j$. The `point-to-plane last-passage time' is defined as
\begin{align}\label{ltdef}
L_t := \max_{Q\in \cq_t} \sum_{i=0}^{t-1}w_{q_i}. 
\end{align}
This model fits into the framework of equation \eqref{kpzevolve2} by taking 
\[
\phi(u,z) = \max_{b\in B^+} u_b + F(z),
\]
where $B^+ = \{e_1,\ldots, e_d\}$ is the set of neighbors of the origin in the positive orthant, and $F$ is a Lipschitz function that transforms the standard Gaussian measure on $\rr$ to the law of the environment in LPP, This allows only a certain class of laws for the environment --- namely, those that can be expressed as Lipschitz functions of Gaussian --- but as noted earlier, this class is quite broad. To see the equivalence with point-to-plane LPP, note that with the above $\phi$, a simple induction shows that
\[
f(t,x) = \max_{Q\in \cq_{t}} \sum_{i=0}^{t-1} F( z_{t-i, x+q_i}),
\]
From this, it is not hard to see that for any $x$ and $t$, $f(t,x)$ has the same law as $L_t$. Indeed, if we define 
\[
w_y := F(z_{t-|y|_1, x+y}),
\]
then $\{w_y\}_{y\in O^+}$ are i.i.d.~random variables, and $f(t,x)=L_t$ if $L_t$ is defined as in~\eqref{ltdef} using these $w$'s.

Note that $\phi$ is equivariant under constant shifts, monotone in the first argument, and Lipschitz continuous in the second argument. Thus, it satisfies all the required conditions. Superconcentration in point-to-plane LPP, with the variance bound $\var(L_t)\le Ct/\log t$ (where $C$ depends only on the dimension and the law of the noise variables) was proved in \cite{chatterjee08, chatterjee14} for $d=2$. \citet{graham12} proved superconcentration in point-to-point directed last-passage percolation in all dimensions, with the same variance bound. I have not seen a proof of superconcentration in point-to-plane directed last passage percolation in $d\ge 3$, but it is possible that it follows from Graham's methods.  The following theorem proves this result, together with a novel subroughness bound that may be of independent interest. 
\begin{thm}\label{lppthm}
Consider the surface $f$ generated by the LPP model defined above. Let $C$ denote any constant that depends only on the dimension $d$ and the law of the environment. Then, we have the following bounds:
\begin{enumerate}
\item[$(1)$] For any $t\ge 2$, $x\in \zz^d$, and $1\le i<j\le d$, $\ee|f(t,x+e_i)-f(t,x+e_j)|\le C\sqrt{\log t}$.
\item[$(2)$] For any $t\ge 2$, $x\in \zz^d$, and $1\le i<j\le d$, $\ee[(f(t,x+e_i)-f(t,x+e_j))^2]\le C\sqrt{t}\log t$.
\item[$(3)$] For any $t\ge 2$ and $x\in \zz^d$, $\var(f(t,x))\le Ct/\log t$. 
\end{enumerate}
\end{thm}
Theorem \ref{lppthm} is proved in Section \ref{lppproof}. The proof of the first claim uses a new argument that is simpler than the proofs of similar claims in~\cite{chatterjee08, chatterjee14, graham12} and gives a better bound. The second claim follows from the first by combining with Theorem~\ref{genvarthm}, and the third claim follows from the second by Theorem \ref{equivthm2}. 

\subsection{Directed polymers}\label{polymersec}
The framework of this paper includes the model of $(d+1)$-dimensional directed polymers in an i.i.d.~random environment~\cite{comets17}, for any $d\ge 1$, as long as the law of the environment can be expressed as the pushforward of the standard Gaussian measure under a  Lipschitz map. Recall that this model is defined as follows. Let $(w_{t,x})_{t\in \zz_{\ge0}, x\in \zz^d}$ be a collection of i.i.d.~random variables, called the `environment', as in LPP. Let $\cp_t$ be the set of all paths of length $t$ started at the origin --- that is, all $P = (p_0,\ldots, p_t)\in (\zz^d)^t$ such that $p_0=0$ and $|p_i-p_{i-1}| = 1$ for all $i\ge 1$. The directed polymer model assigns a random probability measure on $\cp_t$, with a path $P = (p_0,\ldots, p_t)$ assigned a probability proportional to 
\[
\exp\biggl(\beta \sum_{i=0}^{t-1} w_{i,p_i}\biggr),
\]
where $\beta$ is a parameter known as the `inverse temperature' of the model. A key object of interest in this model is  the partition function $Z_t$, defined as
\[
Z_t := \sum_{P\in \cp_t} \exp\biggl(\beta \sum_{i=0}^{t-1} w_{i,p_i}\biggr). 
\]
To capture the logarithm of the partition function of the directed polymer model in our framework, we take
\[
\phi(u,z) = \frac{1}{\beta}\log \biggl(\sum_{b\in B} e^{\beta u_b}\biggr) + F(z),
\]
where $F$ is the Lipschitz map the transforms the standard Gaussian measure into the law of the environment (assuming, as before, that such a map exists). If $f$ is the random surface generated with this driving function and the environment $\mathbf{z}$, and zero initial condition, then a simple induction shows that
\[
f(t,x) = \frac{1}{\beta}\log \biggl[\sum_{P\in \cp_t} \exp\biggl(\beta \sum_{i=0}^{t-1} F(z_{t-i, x+p_i})\biggr)\biggr].
\]
It is not hard to see that $f(t,x)$ has the same law as $\beta^{-1}\log Z_t$. It is also easy to verify that $\phi$ is equivariant under constant shifts, monotone in the first argument, and Lipschitz continuous in the second argument. 

Superconcentration of the log partition function of the directed polymer model, with the bound $\var(\log Z_t)\le Ct/\log t$, was proved by \citet{alexanderzygouras13}. There is a version of this model for $\beta=\infty$, known as the `directed polymer at zero temperature'. For the zero temperature model, superconcentration was proved in~\cite{chatterjee08, chatterjee14} for $d=1$, and in \cite{graham12} for all $d$. 

The following theorem reproves the superconcentration of the log partition function of the directed polymer model at finite $\beta$ using the techniques of this paper, along with a novel subroughness bound that may be of independent interest. The zero temperature case, being very similar to LPP, is omitted. 
\begin{thm}\label{polymerthm}
Consider the surface $f$ generated by the directed polymer model defined above. Let $C$ denote any constant that depends only on the dimension $d$, the inverse temperature $\beta$, and the law of the environment. Then, we have the following bounds:
\begin{enumerate}
\item[$(1)$] For any $t\ge 2$ and $x,y\in \zz^d$ with $|x-y|_1=2$, $\ee|f(t,x)-f(t,y)|\le C$.
\item[$(2)$] For any $t\ge 2$ and $x,y\in \zz^d$ with $|x-y|_1=2$, $\ee[(f(t,x)-f(t,y))^2]\le C\sqrt{t\log t}$.
\item[$(3)$] For any $t\ge 2$ and $x\in \zz^d$, $\var(f(t,x))\le Ct/\log t$. 
\end{enumerate}
\end{thm}
This theorem is proved in Section \ref{polymerproof}. The proof of the first claim uses a new argument. The second claim follows from the first claim and Theorem \ref{genvarthm}, whereas the third claim follows from the second and Theorem \ref{equivthm2}. 

\section{Random walk representation of derivatives}\label{derivsec}
For $1\le s\le t$ and $x,y\in \zz^d$, we will now compute the partial derivative of $f(t,x)$ with respect to $z_{s,y}$, assuming that the driving function is differentiable. It turns out that the derivative is expressible in terms of the transition probabilities of a certain kind of random walk. This random walk representation is crucial for all subsequent analyses. We need the equivariance and monotonicity properties for the proof, but not the Lipschitz property. 

Throughout this section, let $\phi$ be a monotone, equivariant, and differentiable driving function. Writing an element of $\rr^A\times\rr$ as $(u, z)$, where $u = (u_a)_{a\in A}\in \rr^A$ and $z\in \rr$, let $\partial_a \phi$ denote the partial derivative of $\phi(u,z)$ with respect to $u_a$, and let $\partial_z\phi$ denote the partial derivative of $\phi$ with respect to $z$. The following lemma records two important properties of these derivatives, which are consequences of the equivariance and monotonicity properties of $\phi$. 
\begin{lmm}\label{problmm}
For any $(u,z)\in \rr^A \times \rr$, $\partial_a\phi(u,z)\ge 0$ for each $a\in A$, and 
\[
\sum_{a\in A} \partial_a\phi(u,z)=1.
\]
\end{lmm}
\begin{proof}
Fix $(u,z)$. Define a function $g:\rr \to \rr$ as $g(t) := \phi(u+t, z)$ (recall that $u+t$ is the vector obtained by adding $t$ to each coordinate of $u$). By the equivariance of $\phi$, we have that $g(t) = \phi(u,z)+t$. Thus, $g'(t) = 1$ for all $t$. On the other hand, by the definition of $g$,
\[
g'(t) = \sum_{a\in A} \partial_a\phi(u+t,z).
\]
Thus, 
\[
\sum_{a\in A} \partial_a \phi(u,z) = g'(0)=1.
\]
The nonnegativity of $\partial_a \phi(u,z)$ follows from the monotonicity of $\phi$.
\end{proof}



Let $f$ be a growing random surface defined according to \eqref{kpzevolve2}. For any $t\in \zz_{\ge 0}$ and $x\in \zz^d$, define a random walk on $\zz^d$ as follows. The walk starts at $x$ at time $t$, and goes backwards in time, until reaching time $0$. If the walk is at location $y\in \zz^d$ at time $s\ge 1$, then at time $s-1$ it moves to $y+a$ with probability $\partial_a\phi((f(s-1,y+a))_{a\in A},z_{s,y})$, for $a\in A$. By Lemma \ref{problmm}, these numbers are nonnegative and sum to $1$ when summed over $a\in A$. Therefore, this describes a legitimate random walk on $\zz^d$, moving backwards in time. 

\begin{prop}\label{derivprop}
Take any $1\le s\le t$ and $x,y\in \zz^d$. Let $\{S_r\}_{0\le r\le t}$ be the backwards random walk defined above, started at $x$ at time $t$. Then 
\begin{align*}
\fpar{}{z_{s,y}}f(t,x) &= \pp(S_s = y) \partial_z \phi((f(s-1, y+a))_{a\in A}, z_{s,y}). 
\end{align*}
\end{prop}
\begin{proof}
The proof is by induction on $t$. First, suppose that $t=1$. Then $s$ must also be equal to $1$. Moreover, $f(t,x)$ has no dependence on $z_{t,y}$ if $y\ne x$, and so the partial derivative is zero if $y\ne x$. If $y=x$, then by the definition \eqref{kpzevolve2} of $f(t,x)$, it follows that
\begin{align*}
\fpar{}{z_{s,y}}f(t,x) &=  \partial_z \phi((f(t-1, x+a))_{a\in A}, z_{t,x})\\
&= \pp(S_s = x) \partial_z \phi((f(s-1, y+a))_{a\in A}, z_{s,y})
\end{align*}
since $s=t$, $x=y$, and $\pp(S_s = x)=1$. Thus, the claim holds when $t=1$.

Now suppose that the claim has been proved up to time $t-1$. If $s=t$, the proof is the same as in the previous paragraph. So assume that $s<t$. By \eqref{kpzevolve2} and the chain rule for differentiation,
\begin{align*}
\fpar{}{z_{s,y}}f(t,x) &= \sum_{a\in A} \partial_a \phi((f(t-1, x+a))_{a\in A}, z_{t,x}) \fpar{}{z_{s,y}}f(t-1,x+a)\\
&= \sum_{a\in A} \pp(S_{t-1} =x+a) \fpar{}{z_{s,y}}f(t-1,x+a). 
\end{align*}
For each $a\in A$, let $S^a$ be the backwards random walk started at $x+a$ at time $t-1$. Then by the induction hypothesis for time $t-1$,
\begin{align*}
\fpar{}{z_{s,y}}f(t-1,x+a) &=  \pp(S^a_s = y) \partial_z \phi((f(s-1, y+a))_{a\in A}, z_{s,y}). 
\end{align*}
Combining the previous two displays, we get
\begin{align*}
&\fpar{}{z_{s,y}}f(t,x)\\
&= \partial_z \phi((f(s-1, y+a))_{a\in A}, z_{s,y})\sum_{a\in A} \pp(S_{t-1}=x+a) \pp(S^a_s = y).
\end{align*}
But from the definition of the random walks, it is not hard to see that the law of $S^a$ is the same as the law of $S$ given $S_{t-1}=x+a$. Thus,
\begin{align*}
&\sum_{a\in A} \pp(S_{t-1}=x+a) \pp(S^a_s = y) \\
&= \sum_{a\in A} \pp(S_{t-1}=x+a) \pp(S_s = y| S_{t-1}=x+a) = \pp(S_s=y). 
\end{align*}
Combining this with the previous display completes the proof.
\end{proof}

\section{Proof of Theorem \ref{genvarthm}}\label{genvarthmproof}
Let us first prove the theorem under the assumption that $\phi$ is differentiable.
\begin{lmm}\label{genvarlmm}
The conclusions of Theorem \ref{genvarthm} hold if, in addition to the stated hypotheses, we also have that $\phi$ is differentiable.
\end{lmm}
\begin{proof}
Fix $t$ and $x$. Conditioning on the randomness due to the noise variables, let $S= \{S_s\}_{0\le s\le t}$ be the random walk started at $x$ at time $t$ and moving backwards in time, defined in Section \ref{derivsec}. Let $S' = \{S_s'\}_{0\le s\le t}$ be an independent copy of $S$ (conditional on the noise variables). It is not hard to see that $f(t,x)$ is a function of only finitely many of the noise variables. Moreover, by the uniform Lipschitz property, $|\partial_z\phi|$ is uniformly bounded by $L$. Let $\pp'$ denote conditional probability given the noise variables, and let $\ee'$ denote the conditional expectation. Then by Proposition \ref{derivprop} and the above observations, we have
\begin{align*}
\sum_{s=1}^t \sum_{y\in \zz^d} \biggl(\fpar{}{z_{s,y}}f(t,x)\biggr)^2 &\le L^2\sum_{s=1}^t \sum_{y\in \zz^d}(\pp'(S_s = y))^2\\
&= L^2\sum_{s=1}^t \sum_{y\in \zz^d}\pp'(S_s = y, \, S'_s=y)\\
&= L^2\sum_{s=1}^t \pp'(S_s = S_s')\\
&=L^2 \ee'|\{1\le s\le t: S_s = S_s'\}|\le L^2 t. 
\end{align*}
Thus, as a function of the noise variables, $f(t,x)$ is differentiable and Lipschitz with respect to the Euclidean metric, with Lipschitz constant bounded by $L\sqrt{t}$. The claims now follow easily by the Gaussian Poincar\'e inequality and the Gaussian concentration inequality (see \cite[Chapter 2 and Appendix A]{chatterjee14}). 
\end{proof}
To drop the differentiability requirement, several lemmas are needed. Throughout, we work under the hypotheses of Theorem \ref{genvarthm}.
\begin{lmm}\label{liplmm}
The function $\phi$ is Lipschitz with Lipschitz constant $L+1$ with respect to the $\ell^\infty$ norm on $\rr^A \times \rr$. 
\end{lmm}
\begin{proof}
Take any $z\in \rr$ and $u,v\in \rr^A$. For each $a\in A$, let $s_a := \min\{u_a, v_a\}$. Let $s:= (s_a)_{a\in A}$. Let $c:= \max_{a\in A} |u_a-v_a|$. Then $u_a$ and $v_a$ are both in the interval $[s_a, s_a+c]$ for each $a\in A$. Thus, by the monotonicity of $\phi$, $\phi(u,z)$ and $\phi(v,z)$ are both lower bounded by $\phi(s,z)$ and upper bounded by $\phi(s+c, z)$. But by equivariance, $\phi(s+c,z)=\phi(s,z)+c$. This shows that 
\[
|\phi(u,z)-\phi(v,z)| \le c = \|u-v\|_{\ell^\infty}.
\] 
Thus, for any $u,v\in \rr^A$ and $z,z'\in \rr$, we have
\begin{align*}
|\phi(u,z)-\phi(v,z')| &\le |\phi(u,z)-\phi(v,z)| + |\phi(v,z)-\phi(v,z')|\\
&\le \|u-v\|_{\ell^\infty} + L|z-z'|\\
&\le (L+1)\|(u,z)-(v,z')\|_{\ell^\infty},
\end{align*}
which proves the claim. 
\end{proof}

Let $h:\rr^A\times \rr\to [0,\infty)$ be a $C^\infty$ function with compact support, which integrates to $1$. For each $\ve >0$, define the function $h_\ve(x) := \ve^{-d}h(\ve^{-1} x)$. Note that $h_\ve$ is also nonnegative, smooth, and integrates to $1$. Let $\phi_\ve$ be the convolution of $\phi$ with $h_\ve$, that is, for any $x$,
\begin{align}\label{convform}
\phi_\ve(x) = \int h_\ve(x-y) \phi(y) dy = \int \phi(x-y)h_\ve(y)dy.
\end{align}
\begin{lmm}\label{phieplmm}
For any $\ve >0$, $\phi_\ve$ is a differentiable function. Moreover, it has the monotonicity and equivariance properties, and is Lipschitz in the noise variable with Lipschitz constant $L$.
\end{lmm}
\begin{proof}
By Lemma \ref{liplmm}, $\phi$ is Lipschitz. In particular, it is continuous and hence bounded on compact sets. Since $h_\ve$ has compact support, it is now easy to use the first integral in \eqref{convform} and the dominated convergence theorem to deduce that $\phi_\ve$ is differentiable everywhere.  From the second integral in \eqref{convform} and the fact that $h_\ve$ is nonnegative and integrates to $1$, it follows that $\phi_\ve$ is monotone, equivariant, and Lipschitz in the noise variable with Lipschitz constant $L$. 
\end{proof}
Let $f_\ve$ be the growing random surface generated by the driving function $\phi_\ve$, the noise variables $z_{t,x}$, and  initial value $f_\ve(0,x) = f(0,x)$ for all $x$. Combining the above lemma with Lemma~\ref{genvarlmm}, we get the following corollary about $f_\ve$.
\begin{cor}\label{genvarcor}
The conclusions of Theorem \ref{genvarthm} hold for $f_\ve$, for any $\ve >0$.
\end{cor}
\begin{proof}
This is a consequence of Lemma \ref{genvarlmm} and Lemma \ref{phieplmm}, since $\phi_\ve$ satisfies all the conditions of Theorem \ref{genvarthm}, and is moreover differentiable, satisfying the additional criterion demanded by Lemma \ref{genvarlmm}. 
\end{proof}
We also get the following analog of Lemma \ref{liplmm}.
\begin{cor}\label{lipcor}
For any $\ve > 0$, the function $\phi_\ve$ is Lipschitz continuous with Lipschitz constant $L+1$ with respect to the $\ell^\infty$ norm on $\rr^A \times \rr$. 
\end{cor}
\begin{proof}
The proof is exactly the same as the proof of Lemma \ref{liplmm}, after replacing $\phi$ by $\phi_\ve$. This goes through, because by Lemma \ref{phieplmm}, $\phi_\ve$ shares all the relevant properties with $\phi$. 
\end{proof}
Our next goal is to show that $f_\ve$ converges pointwise to $f$ as $\ve \to 0$.  The first step is the following lemma. 
\begin{lmm}\label{uniflmm}
As $\ve \to 0$, $\phi_\ve \to \phi$ uniformly on $\rr^A\times \rr$. 
\end{lmm}
\begin{proof}
Take any $x\in \rr^A\times \rr$. Recall that $h_\ve$ integrates to $1$. Thus, by Lemma \ref{liplmm}, 
\begin{align*}
|\phi_\ve(x)-\phi(x)| &= \biggl|\int h_\ve(x-y) (\phi(y)-\phi(x)) dy\biggr|\\
&\le \int h_\ve(x-y) |\phi(y)-\phi(x)| dy\\
&\le (L+1)\int h_\ve(x-y) \|x-y\|_{\ell^\infty} dy\\
&= (L+1)\int h_\ve(u) \|u\|_{\ell^\infty} du. 
\end{align*}
Now, by the change of variable $v = \ve^{-1}u$, we have
\begin{align*}
\int h_\ve(u) \|u\|_{\ell^\infty} du &= \ve \int h(v) \|v\|_{\ell^\infty} du. 
\end{align*}
Plugging this into the previous display proves the uniform convergence of $\phi_\ve$ to $\phi$ as $\ve \to 0$.
\end{proof}

\begin{lmm}\label{ftconv}
As $\ve \to 0$, $f_\ve(t,x) \to f(t,x)$ for any $t$ and $x$.
\end{lmm}
\begin{proof}
We will prove this by induction on $t$. This is given to be true for $t=0$. Suppose that this holds for $t-1$. Take any $x$. Then by the induction hypothesis for $t-1$,  we have that
\[
\lim_{\ve \to 0} f_\ve(t-1,x+a) = f(t-1, x+a)
\]
for each $a\in A$. By Lemma \ref{uniflmm}, $\phi_\ve \to \phi$ uniformly. By Lemma \ref{liplmm}, $\phi$ is continuous. Combining these three facts, we get
\begin{align*}
\lim_{\ve \to 0} f_\ve(t,x) &= \lim_{\ve \to 0} \phi_\ve((f_\ve(t-1,x+a))_{a\in A}, z_{t,x})\\
&= \phi((f(t-1,x+a))_{a\in A}, z_{t,x}) = f(t,x). 
\end{align*}
This completes the proof of the induction step.
\end{proof}
For $t\in \zz_{\ge 1}$ and $x\in \zz^d$, recall the random walk $\{S_s\}_{0\le s\le t}$ starting at $x$ at time $t$, defined in Section \ref{derivsec}. Let $V_{t,x}$ be the set of all points in $\zz_{\ge 1}\times \zz^d$ that can possibly be accessed by the walk --- that is, the set of all possible values of $(s,S_s)$ as $s$ ranges between $1$ and $t$. Note that for any $t\ge 2$ and $x\in \zz^d$,
\begin{align}\label{slmmeq}
V_{t,x} = \{(t,x)\}\cup \bigcup_{a\in A} V_{t-1, x+a},
\end{align}
and $V_{1,x} = \{(1,x)\}$.  Take any $\ve >0$. Define a new growing surface $g_\ve$, with the same initial values as $f_\ve$ (that is, $f_\ve(0,x)=g_\ve(0,x)$ for all $x$), the same driving function $\phi_\ve$, but the noise field identically equal to zero. Note that $g_\ve$ is a nonrandom function. 
\begin{lmm}\label{bdlmm}
For any $\ve>0$, and any $t$ and $x$, we have
\begin{align*}
|f_\ve(t,x)-g_\ve(t,x)|&\le (L+1)^t \max_{(s,y)\in V_{t,x}} |z_{s,y}|. 
\end{align*}
\end{lmm}
\begin{proof}
The proof is by induction on $t$. For $t=1$, note that by the equality of $f_\ve$ and $g_\ve$ at time $0$, and Lemma~\ref{phieplmm}, we have
\begin{align*}
&|f_\ve(1,x) - g_\ve(1,x)| \\
&= |\phi_\ve((f_\ve(0, x+a))_{a\in A}, z_{1,x}) - \phi_\ve((g_\ve(0,x+a))_{a\in A}, 0)|\\
&\le (L+1)|z_{1,x}|.
\end{align*}
Since $V_{1,x}=\{(1,x)\}$, this proves the claim for $t=1$. Now suppose that it holds for $t-1$. Then by Corollary \ref{lipcor},
\begin{align*}
&|f_\ve(t,x) - g_\ve(t,x)| \\
&= |\phi_\ve((f_\ve(t-1, x+a))_{a\in A}, z_{t,x}) - \phi_\ve((g_\ve(t-1,x+a))_{a\in A}, 0)|\\
&\le (L+1)\max\{\|(f_\ve(t-1, x+a))_{a\in A} - (g_\ve(t-1, x+a))_{a\in A}\|_{\ell^\infty}, |z_{t,x}|\}. 
\end{align*}
But by the induction hypothesis for $t-1$,
\begin{align*}
&\|(f_\ve(t-1, x+a))_{a\in A} - (g_\ve(t-1, x+a))_{a\in A}\|_{\ell^\infty} \\
&= \max_{a\in A} |f_\ve(t-1, x+a) - g_\ve(t-1,x+a)|\\
&\le (L+1)^{t-1} \max_{a\in A} \max_{(s,y)\in V_{t-1,x+a}} |z_{s,y}|.
\end{align*}
The desired result follows by combining the last two displays with~\eqref{slmmeq}. 
\end{proof}
Finally, define another growing surface $g$, with the same initial values as $f$, with driving function $\phi$, and the noise field identically equal to zero.
\begin{lmm}\label{bdlmm2}
For any $t$ and $x$, $g_\ve(t,x) \to g(t,x)$ as $\ve \to 0$. 
\end{lmm}
\begin{proof}
The proof is by induction on $t$. For $t=0$, the result is automatic, since 
\[
g_\ve(0,x) = f_\ve(0,x) = f(0,x) = g(0,x).
\]
Suppose that the claim holds for $t-1$. Then by Lemma \ref{uniflmm},
\begin{align*}
\lim_{\ve \to 0} g_\ve(t,x) &= \lim_{\ve \to 0} \phi_\ve((g_\ve(t-1, x+a))_{a\in A}, 0)\\
&= \phi((g(t-1, x+a))_{a\in A}, 0) = g(t,x),
\end{align*} 
which completes the proof of the lemma.
\end{proof}
We are now ready to prove Theorem \ref{genvarthm}. 
\begin{proof}[Proof of Theorem \ref{genvarthm}]
Take any $t$ and $x$. By Lemma \ref{ftconv}, $f_\ve(t,x) \to f(t,x)$ as $\ve \to 0$. By Lemma \ref{bdlmm}, we see that for any $\ve \in (0,1)$ and any $t$ and $x$,
\begin{align*}
|f_\ve(t,x)| &\le |f_\ve(t,x)-g_\ve(t,x)| + |g_\ve(t,x) - g(t,x)| + |g(t,x)|\\
&\le (L+1)^t \max_{(s,y)\in V_{t,x}} |z_{s,y}| + \sup_{0<\delta < 1} |g_\delta(t,x)-g(t,x)| + |g(t,x)|. 
\end{align*}
By Lemma \ref{bdlmm2}, the middle term on the right is a finite (deterministic) quantity. Also, by the facts that the noise variable are standard Gaussian and that the set $V_{t,x}$ is finite, we see that for any $\theta \ge 0$, the quantity
\begin{align*}
\ee\bigl[\exp\bigl(\theta (L+1)^t \max_{(s,y)\in V_{t,x}} |z_{s,y}|\bigr)\bigr] &\le \ee\biggl[\sum_{(s,y)\in V_{t,x}} \exp(\theta (L+1)^t |z_{s,y}|)\biggr]
\end{align*}
is finite. Thus, the random variables $\{|f_\ve(t,x)|\}_{0<\ve < 1}$ are uniformly bounded by a random variable $M_{t,x}$ such that $\ee(e^{\theta M_{t,x}})$ is finite for any $\theta\ge 0$. Therefore by the dominated convergence theorem, all moments and exponential moments of $f_\ve(t,x)$ converge to the corresponding moments and exponential moments of $f(t,x)$ as $\ve \to 0$. Applying Corollary \ref{genvarcor}, we can now get the required bounds on $\var(f(t,x))$ and $\ee(e^{\theta(f(t,x) - \ee(f(t,x)))})$. The required tail bound follows easily from the bound on the moment generating function. 
\end{proof}

\section{Proof of Theorem \ref{equivthm2}}\label{equivproof}
Throughout this proof, $C$ will denote any positive constant that may only depend on the dimension $d$. The value of $C$ may change from line to line, or even within a line. Fix some $t\ge 1$ and $x,b\in \zz^d$, with $b\ne 0$. Let $\sigma_t^2 := \var(f_{t,x})$ and $\sigma_{b,t}^2 := \ee[(f(t,x) - f(t,x+b))^2]$. Due to the flat initial condition, these  quantities have no dependence on $x$. First, note that by the inequality $(u+v)^2\le 2u^2 + 2v^2$ and the fact that $\ee(f(t,x))$ does not depend on $x$, we have
\begin{align*}
\beta_{b,t} &= \frac{\sigma_{b,t}^2}{4L^2t}\\
&\le \frac{1}{4L^2t}(2\var(f(t,x)) + 2\var(f(t,x+b))) \\
&= \frac{\sigma_t^2}{L^2t} = \alpha_t.
\end{align*}
This proves one of the claimed inequalities. Next, let $k\ge 2$ be an integer, to be chosen later. Let 
\begin{align*}
X := \frac{1}{k}\sum_{i=0}^{k-1} f(t, x+ib).
\end{align*}
For $0\le i\le k-1$, let $\{S^i_s\}_{0\le s\le t}$ be the random walk started at $x+ib$ at time $t$, defined in Section \ref{derivsec}. Let $\pp'$ denote conditional probability given the noise variables. Then by Proposition \ref{derivprop}, for any $1\le s\le t$ and $y\in \zz^d$,
\begin{align}\label{xderiv}
\fpar{X}{z_{s,y}} &=  \partial_z \phi((f(s-1, y+a))_{a\in A}, z_{s,y}) \frac{1}{k}\sum_{i=0}^{k-1} \pp'(S^i_s = y).
\end{align}
Consequently, 
\begin{align}\label{l1bd}
\biggl\|\fpar{X}{z_{s,y}}\biggr\|_{L^1} &\le \frac{L}{k}\sum_{i=0}^{k-1}\pp(S^i_s=y),  
\end{align}
where $\|Z\|_{L^1}$ denotes the $L^1$ norm of a random variable $Z$. 
Due to the flat initial condition, the law of $f$ is invariant under spatial translations, which implies that $\pp(S^i_s = y) = \pp(S^0_s = y - ib)$.  Thus, 
\begin{align*}
\biggl\|\fpar{X}{z_{s,y}}\biggr\|_{L^1} &\le \frac{L}{k}\sum_{i=0}^{k-1}\pp(S^0_s=y-ib). 
\end{align*}
Let $B_{s,y}$ denote the quantity on the right. Now, again by \eqref{xderiv}, 
\begin{align*}
\biggl(\fpar{X}{z_{s,y}}\biggr)^2 &\le L^2\biggl(\frac{1 }{k}\sum_{i=0}^{k-1} \pp'(S^i_s = y)\biggr)^2\\
&\le  \frac{L^2 }{k}\sum_{i=0}^{k-1} \pp'(S^i_s = y).
\end{align*}
This shows that
\begin{align*}
\biggl\|\fpar{X}{z_{s,y}}\biggr\|_{L^2}^2 &\le A_{s,y}^2, 
\end{align*}
where $A_{s,y} := \sqrt{LB_{s,y}}$. But by \eqref{l1bd}, 
\begin{align*}
\biggl\|\fpar{X}{z_{s,y}}\biggr\|_{L^1}  &\le B_{s,y} = A_{s,y} \sqrt{\frac{B_{s,y}}{L}},
\end{align*}
which can be rewritten as 
\[
\frac{A_{s.y}}{\|\partial X/\partial z_{s,y}\|_{L^1}} \ge \sqrt{\frac{L}{B_{s,y}}}. 
\]
Lastly, note that the events $S^0_s  = y-ib$ are disjoint as $i$ varies, which shows that $B_{s,y}$ is bounded above by $L/k$. Thus, by Talagrand's $L^1$-$L^2$ inequality (specifically, the version displayed in \cite[Theorem 5.1]{chatterjee14}), we get
\begin{align*}
\var(X) &\le C\sum_{s=1}^t \sum_{y\in \zz^d} \frac{A_{s,y}^2}{1+\log(A_{s,y}/\|\partial X/\partial z_{s,y}\|_{L^1})} \\
&\le \frac{C}{\log k}\sum_{s=1}^t \sum_{y\in \zz^d} A_{s,y}^2=  \frac{CL^2}{k\log k}\sum_{s=1}^t \sum_{y\in \zz^d} \sum_{i=0}^{k-1} \pp(S^0_s = y-ib)\\
&= \frac{CL^2}{\log k}\sum_{s=1}^t \sum_{v\in \zz^d} \pp(S^0_s = v)= \frac{CL^2t}{\log k}. 
\end{align*}
Now note that
\begin{align*}
\|X-f(t,x)\|_{L^2} &\le \frac{1}{k}\sum_{i=1}^{k-1}\|f(t,x+ib) - f(t,x)\|_{L^2}\\
&\le \frac{1}{k}\sum_{i=1}^{k-1}\sum_{j=0}^{i-1}\|f(t,x+(j+1)b) - f(t,x+jb)\|_{L^2}\\
&= \frac{1}{k}\sum_{i=1}^{k-1}\sum_{j=0}^{i-1} \sigma_{b,t}\le \frac{k\sigma_{b,t}}{2}. 
\end{align*}
Since $\ee(X) = \ee(f(t,x))$, the last two displays show that 
\begin{align}
\sigma_t^2 &= \var(f(t,x)) \notag\\
&\le 2\ee[(f(t,x) - X)^2] + 2\var(X)\notag\\
&\le \frac{k^2\sigma_{b,t}^2}{2} + \frac{2CL^2t}{\log k}. \label{kineq}
\end{align}
By Theorem \ref{genvarthm} and the inequality $\beta_{b,t}\le \alpha_t$, we get that $\beta_{b,t}\le\alpha_t\le 1$. So, if $\beta_{b,t}\ge 1/10$, then the bound  $\alpha_t \le C/|\log \beta_{b,t}|$ is trivial. Let us assume that $\beta_{b,t}<1/10$. Then choosing $k$ to be the integer part of $(\beta_{b,t}|\log\beta_{b,t}|)^{-1/2}$ and using \eqref{kineq}, we get
\begin{align*}
\sigma_t^2 &\le \frac{CL^2t}{|\log\beta_{b,t}|},
\end{align*}
which is the same as $\alpha_t \le C/|\log \beta_{b,t}|$.

\section{Proof of Theorem \ref{rsosthm}}\label{rsosproof}
The growth mechanism for $f$ does not directly fit into the framework of this paper, since the heights at even and odd sites are updated alternately. However, this can be easily taken care of, as follows. Let $g$ be another growing random surface, with the same growth mechanism as $f$, except that the height at every site is updated at each step. That is, we start with $g(0,\cdot)\equiv 0$, and for each $t$ and $x$, we choose $g(t+1,x)$ uniformly from the interval
\[
[\max_{b\in B} g(t,x+b)-1, \min_{b\in B} g(t,x+b)+1].
\]
(It is not hard to prove by induction that this interval is always nonempty. To see this, suppose that this is true up to time $t-1$. Then, by the construction of $g(t,x+b)$ according to the above rule, we see that $|g(t, x+b)-g(t-1,x)|\le 1$. Since this holds for each $b$, the above interval must be nonempty.) 

Next, define $h(0,x) := 0$ for all $x$, and for $t\ge 1$, let 
\[
h(t,x) :=
\begin{cases}
g(t-1,x) &\text{ if $t$ and $x$ have the same parity,}\\
g(t,x) &\text{ otherwise.}
\end{cases}
\]
We claim that $h$ has the same law as $f$, and in fact, the same growth mechanism. (It is important to note that this is true only because $f(t+1,x)$ is determined by $(f(t,x+b))_{b\in B}$ and not $(f(t,x+a))_{a\in A}$ in this model.) To see this, take any $t\ge 0$ and $x\in \zz^d$. Suppose that $t$ and  $x$ are both even. Then by the above definition, $h(t+1,x) = g(t+1,x)$. By the definition of $g$, $g(t+1,x)$ is chosen uniformly from the interval
\[
[\max_{b\in B} g(t,x+b)-1, \min_{b\in B} g(t,x+b)+1].
\]
But $g(t,x+b) = h(t,x+b)$ for each $b\in B$. Thus, $h(t+1,x)$ is chosen uniformly from the interval
\[
[\max_{b\in B} h(t,x+b)-1, \min_{b\in B} h(t,x+b)+1].
\]
Next, suppose that $t$ is even and $x$ is odd. Then $h(t+1,x) = g(t,x)$. But in this case, we also have $g(t,x)=h(t,x)$. Thus, $h(t+1,x)=h(t,x)$.  This shows that the growth of $h$ is governed by the same rule as that for $f$ at even times. A similar argument shows that this is also true at odd times. 

Since $h$ has the same law as $f$, it suffices to obtain the required variance bound for $h(t,x)$. This, on the other hand, holds if a similar bound holds for the variance of $g(t,x)$, because $h(t,x)$ is equal to either $g(t,x)$ or $g(t-1,x)$, deterministically depending on $t$ and $x$. We will show this using Theorem \ref{equivthm2}. There are two steps in showing this. First, we have to show that the growth of $g$ is governed by the equation \eqref{kpzevolve2} for some suitable function $\phi$ that has the monotonicity and equivariance properties, and is Lipschitz in the noise variable. The second step is to show that $g$ is subrough, with a suitable quantitative bound.

We will actually carry out the second step first. Since $h$ has the same growth mechanism as $f$, it satisfies the constraint that $|h(t,x)-h(t,y)|\le1$ for any two neighboring points $x$ and $y$. Thus, we have that for any $t$ and any $b,b'\in B$, 
\begin{align*}
|h(t,x)- h(t,x+b+b')|\le 2. 
\end{align*}
This shows that if $t$ and $x$ have opposite parities, then for any $b,b'\in B$,
\begin{align}\label{g1}
|g(t,x)-g(t,x+b+b')| &= |h(t,x)-h(t,x+b+b')|\le 2. 
\end{align}
A similar argument proves that the above bound also holds if $t$ and $x$ have the same parity. The details are as follows. Define $\tilh(0,x) := 0$ for all $x$, and for $t\ge 1$, let 
\[
\tilh(t,x) :=
\begin{cases}
g(t,x) &\text{ if $t$ and $x$ have the same parity,}\\
g(t-1,x) &\text{ otherwise.}
\end{cases}
\]
Then by a similar argument as for $h$, it follows that $\tilh$ grows as follows:
\begin{itemize}
\item If $t$ is even, then for each odd vertex $x$, $\tilh(t+1,x)$ is chosen uniformly from the interval 
\[
[\max_{b\in B} \tilh(t,x+b) -1, \min_{b\in B} \tilh(t,x+b)+1], 
\]
and for each even vertex $x$, $\tilh(t+1,x)=\tilh(t,x)$. 
\item If $t$ is odd, the update rules for odd and even vertices are switched in the above step. 
\end{itemize}
This shows that $\tilh$ also satisfies the constraint that $|\tilh(t,x)-\tilh(t,y)|\le 1$ for any two neighboring points $x$ and $y$. From this, it follows that when $t$ and $x$ have the same parity, then for any $b,b'\in B$,
\begin{align}\label{g2}
|g(t,x)-g(t,x+b+b')| &= |\tilh(t,x)-\tilh(t,x+b+b')|\le 2. 
\end{align}
This completes the proof of the subroughness of $g$, and in fact, gives the quantitative bound
\begin{equation}\label{gsubrough}
\ee[(g(t,x) - g(t,x+2e_1))^2] \le 4. 
\end{equation}
Let us now show that the growth of $g$ is indeed governed by \eqref{kpzevolve2} with a driving function $\phi$ that is monotone, equivariant, and Lipschitz in the noise variable. Let $z_{t,x}$ be i.i.d.~standard Gaussian random variables. Let $\Phi$ be the standard Gaussian c.d.f., so that $\Phi(z_{t,x})$ are i.i.d.~Uniform$[0,1]$ random variables. Then by the definition of $g$, we can express $g(t+1,x)$ as 
\begin{align*}
g(t+1,x) &= \Phi(z_{t+1,x})(\max_{b\in B} g(t,x+b)-1) \\
&\qquad+ (1-\Phi(z_{t+1,x}))(\min_{b\in B} g(t,x+b)+1)\\
&= \Phi(z_{t+1,x}) (\max_{b\in B} g(t,x+b) - \min_{b\in B} g(t,x+b)) \\
&\qquad + \min_{b\in B} g(t,x+b) + 1-2\Phi(z_{t+1,x}). 
\end{align*}
Take any $t$ and $x$, and any $b,b'\in B$. Then $-b\in B$, and so, by \eqref{g1} and \eqref{g2},
\begin{align*}
|g(t,x+b) - g(t,x+b')| &= |g(t,x+b) - g(t, x+b +b'-b)|\le 2. 
\end{align*}
This shows that 
\begin{align*}
0\le \max_{b\in B} g(t,x+b) - \min_{b\in B} g(t,x+b) \le 2.
\end{align*}
So, if we define a function $\xi:\rr\to \rr$ as 
\begin{align*}
\xi(a) &= 
\begin{cases}
a &\text{ if } 0\le a\le 2,\\
2 &\text{ if } a>2,\\
0 &\text{ if } a < 0,
\end{cases}
\end{align*}
and define $\phi:\rr^A\times \rr\to \rr$ as 
\[
\phi(u,z) = \Phi(z) \xi\bigl(\max_{b\in B} u_b -\min_{b\in B} u_b\bigr) + \min_{b\in B} u_b + 1 - 2\Phi(z), 
\]
then the growth of $g$ is governed by \eqref{kpzevolve2} with driving function $\phi$.

Take any $u\in \rr^A$ and $z\in \rr$. Suppose that one coordinate of $u$ is increased by some positive amount. Then either $\min_{b\in B} u_b$ remains the same, in which case $\phi(u,z)$ cannot decrease; or $\min_{b\in B} u_b$ increases by some amount $\ve$. In the latter case, $\max_{b\in B} u_b -\min_{b\in B} u_b$ cannot decrease by more than $\ve$. Since the slope of $\xi$ is everywhere bounded by $1$, in this case $\phi(u,z)$ increases by at least $(1-\Phi(z))\ve$. This shows that $\phi$ is monotone in its first argument. Equivariance under constant shifts is clear from the definition of $\phi$. Lastly, note that
\begin{align*}
\fpar{\phi}{z} &= \Phi'(z) \xi\bigl(\max_{b\in B} u_b - \min_{b\in B} u_b\bigr) - 2\Phi'(z). 
\end{align*}
Since $\xi(a)\in[0,2]$ for all $a\in \rr$ and $\Phi'$ is uniformly bounded by $1/\sqrt{2\pi}$,  this shows that
\[
\biggl|\fpar{\phi}{z}\biggr| \le \frac{4}{\sqrt{2\pi}}. 
\]
Thus, we may indeed apply Theorem \ref{equivthm2} to the surface $g$. By the estimate \eqref{gsubrough}, this completes the proof.

\section{Proof of Theorem \ref{ballthm1}}\label{ballisticproof}
The key step in the proof is to show that moving maxima in stationary random fields cannot fluctuate wildly. We start with the following simple lemma.
\begin{lmm}\label{numlmm}
Let $1\le r\le k$ be two integers, and let $x_0,x_1,\ldots, x_{k+r}$ be real numbers. For $0\le i\le r$, let $m_i := \max\{x_i,x_{i+1},\ldots, x_{i+k}\}$.  Then there is some $0\le i^*\le r$ such that $m_0\ge m_1\ge \cdots \ge m_{i^*}$ and $m_{i^*}\le m_{i^*+1}\le \cdots \le m_r$. 
\end{lmm}
\begin{proof}
Suppose that $m_i < m_{i+1}$ for some $0\le i< r$. Since we have $m_i = \max\{x_i,\ldots,x_{i+k}\}$ and $m_{i+1} =\max\{x_{i+1},\ldots,x_{i+k+1}\}$, this is possible only if $m_{i+1}=x_{i+k+1}$. Take any $i+1\le j\le r$. Since $r\le  k$, we have 
\[
j\le r \le k \le i+k+1.
\]
On the other hand, since $i+1\le j$, we have
\[
i+k+1\le j+k. 
\]  
Thus, $i+k+1$ lies between $j$ and $j+k$, and hence
\begin{align*}
m_j = \max\{x_{j},\ldots,x_{j+k}\} \ge x_{i+k+1} = m_{i+1}.
\end{align*}
So, we have shown that if the sequence $m_0,m_1,\ldots,m_r$ has a strict increase from $m_i$ to $m_{i+1}$, it can never go down below $m_{i+1}$ subsequently. It is easy to see that this proves the claim.
\end{proof}
\begin{cor}\label{simpcor}
Let $x_i$ and $m_i$ be as in Lemma \ref{numlmm}. Then 
\begin{align*}
\sum_{i=0}^{r-1} |m_i - m_{i+1}| \le 2\max_{0\le i,j\le k+r} |x_i-x_j|. 
\end{align*}
\end{cor}
\begin{proof}
By Lemma \ref{numlmm}, there is some $0\le i^*\le r$ such that $m_0\ge m_1\ge \cdots \ge m_{i^*}$ and $m_{i^*}\le m_{i^*+1}\le \cdots \le m_r$. Therefore,
\begin{align*}
\sum_{i=0}^{r-1} |m_i - m_{i+1}| &= \sum_{i=0}^{i^*-1} (m_i - m_{i+1}) + \sum_{i=i^*}^{r-1} (m_{i+1}-m_i)\\
&= m_0 - m_{i^*} + m_r - m_{i^*}. 
\end{align*}
But clearly, $m_0 - m_{i^*} $ and $m_r - m_{i^*}$ are both bounded above by the maximum value of $|x_i-x_j|$ over all $0\le i,j\le k+r$. This completes the proof. 
\end{proof}
Let $(g(x))_{x\in \zz^d}$ be any random field whose law is invariant under translations. For each $s>0$, let 
\begin{align*}
\mu(s) := \ee\biggl(\max_{|x|_1\le s, |y|_1\le s} |g(x)-g(y)|\biggr),
\end{align*}
and assume that this quantity is finite. Here $|x|_1$ denotes the $\ell^1$ norm of $x$.
Let $D$ be a finite subset of $\zz^d$ and $x_0$ be a point in $\zz^d$. For each $i\ge 0$, let $D_i := D + ix_0$ be the translate of $D$ by $ix_0$. Let 
\[
X_i := \max_{x\in D_i} g(x).
\]
Given some large $k$, the following lemma shows that $X_0$ is unlikely to be larger than the maximum of $X_1,\ldots,X_{k+1}$. This is not surprising since the random field is stationary; the point of the lemma is that it gives a quantitative bound under minimal assumptions.
\begin{lmm}\label{maxlmm}
Let all notation be as above. Let $s$ be the sum of the $\ell^1$ diameter of $D$ and $2k|x_0|_1$. Then
\begin{align*}
\ee[(X_0 - \max\{X_1,\ldots, X_{k+1}\})^+] &\le \frac{2\mu(s)}{k},
\end{align*}
where $a^+$ denotes the positive part of a real number $a$. 
\end{lmm}
\begin{proof}
For each $i\ge 0$, let $M_i := \max\{X_i,X_{i+1},\ldots,X_{i+k}\}$. By Corollary \ref{simpcor},
\begin{align*}
\sum_{i=0}^{k-1} |M_i-M_{i+1}| &\le 2\max_{0\le i,j\le 2k} |X_i-X_j|. 
\end{align*}
By translation invariance, $\ee|M_i-M_{i+1}|$ is the same for each $i$. Thus, the above inequality gives
\begin{align*}
\ee|M_0-M_1|&\le \frac{2}{k}\ee\biggl(\max_{0\le i,j\le 2k} |X_i-X_j|\biggr). 
\end{align*}
Without loss of generality, suppose that $0\in D$. Then each point in the union of $D_0,\ldots, D_{2k}$ has $\ell^1$ norm bounded by $s$. Hence, the expectation on the right side of the above inequality is bounded by $\mu(s)$. Lastly, note that 
\begin{align*}
|M_0-M_1| &\ge (M_0-M_1)^+ \ge (X_0-M_1)^+. 
\end{align*}
Thus, $\ee[(X_0-M_1)^+] \le 2\mu(s)/k$, which is what we wanted to prove.
\end{proof}
For each $r\ge 0$, let $G_r := \max_{|x|_1\le r} g(x)$. The following lemma gives an upper bound on the growth rate of $G_r$. The proof uses Lemma \ref{maxlmm}. 
\begin{lmm}\label{glmm}
For any $r\ge 4d$, we have 
\[
\ee|G_{r+1}-G_r|\le \frac{8d^2\mu(6r)}{r}.
\]
\end{lmm}
\begin{proof}
In the following, we will denote the coordinates of any vector $x\in \zz^d$ be $x_1,\ldots,x_d$. Take any $r\ge 4d$. For $i=1,\ldots,d$, define 
\begin{align*}
A_i^+ &:= \{x: |x|_1 = r+1, \, |x_i| \ge |x_j| \text{ for all $1\le j\le d$, and } x_i \ge 0\},\\
A_i^- &:= \{x: |x|_1 = r+1, \, |x_i| \ge |x_j| \text{ for all $1\le j\le d$, and } x_i \le 0\}.
\end{align*}
Note that any $x$ with $|x|_1=r+1$ must belong to $A_i^+$ or $A_i^-$ for at least one $i$. 

Now, let $D := A_1^+$. Note that for any $x\in D$, 
\begin{align*}
x_1 &\ge \frac{1}{d}\sum_{i=1}^d |x_i| = \frac{r+1}{d}. 
\end{align*}
For each $i$, let $D_i := D - ie_1$, and let $X_i := \max_{x\in D_i} g(x)$. Let $k := [r/d]-1$. The above inequality shows that for any $x\in D$ and $y= x-ie_1$ for some $1\le i\le k+1$, we have $x_1> y_1 \ge 0$ and $y_i=x_i$ for $i\ne 1$. Thus, $|y|_1\le r$. This shows that the sets $D_1,\ldots,D_{k+1}$ are all subsets of the $\ell^1$ ball of radius $r$ around the origin. Lastly, note that the $\ell^1$ diameter of $D$ is bounded above by $2(r+1)$. Therefore by Lemma~\ref{maxlmm}, we get
\begin{align*}
\ee[(X_0 - G_r)^+] &\le \ee[(X_0-\max\{X_1,\ldots,X_{k+1}\})^+] \le \frac{2\mu(2r+2 +2k)}{k}.
\end{align*}
The same upper bound holds if we take $D=A_i^+$ or $D=A_i^-$ for any $i$. Thus, defining
\[
Y_i := \max_{x\in A_i^+} g(x), \ \ Z_i := \max_{x\in A_i^-} g(x),
\]
we have
\begin{align*}
\ee|G_{r+1}-G_r| &= \ee[(\max\{Y_1,\ldots,Y_d, Z_1,\ldots,Z_d\} - G_r)^+]\\
&\le \sum_{i=1}^d (\ee[(Y_i - G_r)^+] + \ee[(Z_i-G_r)^+])\\
&\le \frac{4d}{k}\mu(2r+2+2k). 
\end{align*}
The proof is completed by observing that $k = [r/d]-1\ge r/d -2 \ge r/2d$ (since $r\ge 4d$), and $\mu(2r+2+2k)\le \mu(6r)$, since $\mu$ is an increasing function and $2k+2\le 2r+2\le 4r$. 
\end{proof}
We now specialize to random surfaces generated according to \eqref{kpzevolve2} with flat initial condition. Note that if $f$ is such a growing surface, the field $f(t,\cdot)$ is a translation invariant random field at each time $t$. Henceforth,  $C$ will denote any constant that depends only on $\phi$ and $d$. 
\begin{lmm}\label{fdiff1}
Now let $f$ be a growing random surface generated by a driving function that is monotone, equivariant, and Lipschitz in the noise variable, with initial condition $f(0,\cdot)\equiv 0$, and i.i.d.~standard Gaussian noise field. Then for any $t\ge1$ and $r\ge 4d$,
\[
\ee\biggl|\max_{|x|_1\le r} f(t,x)- \max_{|x-e_1|_1\le r} f(t,x)\biggr| \le  \frac{\sqrt{Ct \log (Cr^d)}}{r},
\]
where $C$ is a constant that depends only on $\phi$ and $d$. 
\end{lmm}
\begin{proof}
Take any $\theta \in \rr$. By translation invariance and Theorem \ref{genvarthm}, we have that for any $x$ and $y$,
\begin{align*}
\ee(e^{\theta(f(t,x) - f(t,y))}) &\le \sqrt{\ee(e^{2\theta(f(t,x) - \ee(f(t,x)))}) \ee(e^{2\theta(f(t,y) - \ee(f(t,y)))}) } \le e^{Ct\theta^2}. 
\end{align*}
Consequently, for any $\theta>0$,
\begin{align*}
\ee(e^{\theta|f(t,x) - f(t,y)|}) &\le \ee(e^{\theta(f(t,x) - f(t,y))}) + \ee(e^{-\theta(f(t,x) - f(t,y))}) \le 2e^{Ct\theta^2}. 
\end{align*}
Thus, for any $r\ge 1$ and $\theta>0$,
\begin{align*}
&\ee\biggl(\max_{|x|_1\le r, |y|_1\le r} |f(t,x)-f(t,y)|\biggr) \\
&=\frac{1}{\theta} \ee\biggl[\log \exp\biggl(\theta\max_{|x|_1\le r, |y|_1\le r} |f(t,x)-f(t,y)|\biggr)\biggr]\\
&\le\frac{1}{\theta} \ee\biggl[\log \sum_{|x|_1\le r, |y|_1\le r}e^{\theta|f(t,x)-f(t,y)|}\biggr]\\
&\le \frac{1}{\theta} \log \sum_{|x|_1\le r, |y|_1\le r}\ee(e^{\theta|f(t,x)-f(t,y)|})\le \frac{\log (Cr^d)}{\theta} + Ct\theta. 
\end{align*}
Optimizing over $\theta$, we get
\begin{align*}
\ee\biggl(\max_{|x|_1\le r, |y|_1\le r} |f(t,x)-f(t,y)|\biggr)  &\le \sqrt{Ct \log (Cr^d)}. 
\end{align*}
Thus, by Lemma \ref{glmm} (with $g(\cdot) = f(t,\cdot)$), we get that for any $r\ge 4d$, 
\begin{align}\label{maxf}
\ee\biggl|\max_{|x|_1\le r} f(t,x)- \max_{|x|_1\le r+1} f(t,x)\biggr| &\le \frac{\sqrt{Ct \log (Cr^d)}}{r}. 
\end{align}
For any $x$ such that $|x-e_1|_1\le r$, we have $|x|_1\le r+1$. Thus, the above inequality gives
\begin{align}
&\ee\biggl[\biggl(\max_{|x-e_1|_1\le r} f(t,x) - \max_{|x|_1\le r} f(t,x)\biggr)^+\biggr] \notag\\
&\le \ee\biggl[\biggl(\max_{|x|_1\le r+1} f(t,x) - \max_{|x|_1\le r} f(t,x)\biggr)^+\biggr] \notag\\
&\le \frac{\sqrt{Ct \log (Cr^d)}}{r}. \label{maxineq1}
\end{align}
Now, applying translation invariance to \eqref{maxf}, we have
\begin{align}\label{maxf2}
\ee\biggl|\max_{|x-e_1|_1\le r} f(t,x)- \max_{|x-e_1|_1\le r+1} f(t,x)\biggr| &\le \frac{\sqrt{Ct \log (Cr^d)}}{r}. 
\end{align}
For any $x$ such that $|x|_1\le r$, we have $|x-e_1|_1\le r+1$. Thus, by \eqref{maxf2}, 
\begin{align}
&\ee\biggl[\biggl(\max_{|x|_1\le r} f(t,x) - \max_{|x-e_1|_1\le r} f(t,x)\biggr)^+\biggr] \notag \\
&\le \ee\biggl[\biggl(\max_{|x-e_1|_1\le r+1} f(t,x) - \max_{|x-e_1|_1\le r} f(t,x)\biggr)^+\biggr] \notag \\
&\le \frac{\sqrt{Ct \log (Cr^d)}}{r}. \label{maxineq2}
\end{align}
Combining \eqref{maxineq1} and \eqref{maxineq2}, we get the desired inequality.
\end{proof}
Henceforth, let $f$ be a growing random surface generated by a driving function of max type (satisfying \eqref{maxtype}), with initial condition $f(0,\cdot)\equiv 0$, and i.i.d.~standard Gaussian noise field.
\begin{lmm}\label{fdiff2}
For any $1\le r\le t$ and any $x\in \zz^d$,
\begin{align*}
\biggl|f(t,x) - \max_{|y|_1\le r} f(t-r,x+y)\biggr| &\le \sum_{k=0}^{r-1} \max_{|y|_1\le k}(K_1+K_2 |z_{t-k, x+y}|). 
\end{align*}
\end{lmm}
\begin{proof}
Fix some $t\ge 1$ and $x\in \zz^d$. The proof will be by induction on $r$. Note that by \eqref{kpzevolve2} and \eqref{maxtype}, 
\begin{align}\label{freq}
\biggl|f(t,x) - \max_{a\in A} f(t-1, x+a)\biggr| &\le K_1 + K_2|z_{t,x}|. 
\end{align}
This proves the claim for $r=1$. Now suppose that the claim is true up to $r-1$. Then, for any $a\in A$, 
\begin{align}
&\biggl|f(t-1, x+a) - \max_{|y|_1\le r-1} f(t-r, x+a+y)\biggr| \notag\\
&\le \sum_{k=0}^{r-2} \max_{|y|_1\le k} (K_1+K_2|z_{t-k-1, x+a+y}|). \label{maxinduc}
\end{align}
Now, as $a$ ranges over $A$ and $y$ ranges over the $\ell^1$ ball with radius $r-1$ centered at $0$, the sum $a+y$ ranges over the $\ell^1$ ball with radius $r$ centered at $0$. Thus, 
\begin{align*}
&\biggl|\max_{a\in A} f(t-1, x+a) - \max_{|y|_1\le r} f(t-r, x+y)\biggr| \\
&= \biggl|\max_{a\in A} f(t-1, x+a) - \max_{a\in A} \max_{|y|_1\le r-1} f(t-r, x+a+y)\biggr| \\
&\le \max_{a\in A} \biggl|f(t-1, x+a) - \max_{|y|_1\le r-1} f(t-r, x+a+y)\biggr|.
\end{align*}
Combining this with \eqref{maxinduc}, we get
\begin{align*}
&\biggl|\max_{a\in A} f(t-1, x+a) - \max_{|y|_1\le r} f(t-r, x+y)\biggr| \\
&\le \max_{a\in A} \sum_{k=0}^{r-2} \max_{|y|_1\le k} (K_1+K_2|z_{t-k-1, x+a+y}|)\\
&\le \sum_{k=0}^{r-2} \max_{a\in A}\max_{|y|_1\le k} (K_1+K_2|z_{t-k-1, x+a+y}|)\\
&= \sum_{k=0}^{r-2} \max_{|y|_1\le k+1} (K_1+K_2|z_{t-k-1, x+y}|). 
\end{align*}
Finally, combining this with \eqref{freq} completes the induction step.
\end{proof}
Combining Lemma \ref{fdiff1} and Lemma \ref{fdiff2} yields the following bound on the expected absolute difference between the heights at neighboring sites.
\begin{lmm}\label{fdiffmain}
For any $t\ge 2$ and $x\in \zz^d$,
\begin{align*}
\ee|f(t,x)-f(t,x+e_1)|&\le Ct^{1/4}\sqrt{\log t}. 
\end{align*}
\end{lmm}
\begin{proof}
Take any  $2\le r\le t$. Define
\begin{align*}
M_1 := \max_{|y|_1\le r} f(t-r,x+y), \ \ M_2 := \max_{|y|_1\le r} f(t-r,x+e_1+y). 
\end{align*}
Then by Lemma \ref{fdiff2} and a standard estimate for Gaussian random variables,
\begin{align*}
\ee|f(t,x)-M_1| &\le  \sum_{k=0}^{r-1} \ee\biggl(\max_{|y|_1\le k}(K_1+K_2 |z_{t-k, x+y}|)\biggr)\\
&\le C\sum_{k=0}^{r-1}(1+\sqrt{\log(k+1)})\\
&\le Cr\sqrt{\log r}. 
\end{align*}
By translation invariance, the same bound holds for $\ee|f(t,x+e_1)-M_2|$. On the other hand, by Lemma \ref{fdiff1} and translation invariance,
\begin{align*}
\ee|M_1-M_2|&\le \frac{\sqrt{C(t-r)\log (Cr^d)}}{r}. 
\end{align*}
Combining, and choosing $r = [t^{1/4}]$, we get the desired result. 
\end{proof}
We are now ready to complete the proof of Theorem \ref{ballthm1}.
\begin{proof}[Proof of Theorem \ref{ballthm1}]
The first claim is already proved by Lemma~\ref{fdiffmain}. For the second claim, let us assume without loss of generality that $y=x+e_1$. Let 
\[
D := |f(t,x)-f(t,y)|.
\]
By translation invariance, $\ee(f(t,x)) = \ee(f(t,y))$. Therefore, by Theorem \ref{genvarthm},
\begin{align*}
\pp(|D|\ge r) \le 4 e^{-Cr^2/t}
\end{align*}
for all $r\ge 0$. On the other hand, by the first claim of the theorem, 
\[
\pp(|D|\ge r) \le \frac{\ee|D|}{r}\le \frac{Ct^{1/4}\sqrt{\log t}}{r}. 
\]
Thus, for any $K$,
\begin{align*}
&\ee[(f(t,x)-f(t,y))^2] = \int_0^\infty 2r \pp(|f(t,x)-f(t,y)|\ge r) dr\\
&\le  \int_0^\infty C_1r \min\biggl\{\frac{t^{1/4}\sqrt{\log t}}{r},  e^{- C_2r^2/t}\biggr\} dr\\
&\le \int_0^K C_1 t^{1/4}\sqrt{\log t} dr + \int_K^\infty C_1 r e^{-C^2r^2/t} dr\\
&= C_1 K t^{1/4}\sqrt{\log t} + C_3 t e^{-C_2 K^2/t}.
\end{align*}
Choosing $K = C_4 \sqrt{t\log t}$ for some sufficiently large $C_4$ completes the proof of the second claim. 
The third claim follows from the second by Theorem \ref{equivthm2}. 
\end{proof}

\section{Proof of Theorem \ref{lppthm}}\label{lppproof}
Throughout this proof, $C, C_1, C_2, \ldots$ will denote constants that depend only on the dimension and the law of the noise variables. Let $\cq_t$ be as in Subsection \ref{lppsec}. Let us assume that the law of the environment has  mean zero, since it is not hard to see that this does not cause any loss of generality. We need a simple lemma about real numbers.
\begin{lmm}\label{lppsimple}
Let $x_1,\ldots,x_n$ be real numbers. Then 
\[
\max_{1\le i\le n} x_i - \frac{1}{n}\sum_{i=1}^n x_i \ge \frac{2}{n^3}\sum_{1\le i<j\le n}|x_i-x_j|.
\]
\end{lmm}
\begin{proof}
Without loss of generality, suppose that $x_1$ is the maximum of the numbers, and $x_2$ is the minimum. Then
\begin{align*}
\max_{1\le i\le n} x_i - \frac{1}{n}\sum_{i=1}^n x_i  &= \frac{1}{n}\sum_{i=1}^n (x_1 - x_i) \ge \frac{x_1-x_2}{n}.
\end{align*}
But $x_1 - x_2 \ge |x_i-x_j|$ for all $i$ and $j$, and hence, $x_1-x_2\ge $ the average of $|x_i-x_j|$ over all distinct $i$ and $j$. This proves the claim.
\end{proof}

Since $\ee(F(z))=0$ for a standard Gaussian random variable $z$, there must exist $u,v\in \rr$ such that $F(u)\le 0$ and $F(v)\ge 0$. By the continuity of $F$, it follows that there exists $u^*\in \rr$ where $F(u^*)=0$. Define a new surface $\tf$ by replacing $z_{1,x}$ by $u^*$ for all $x$, but keeping all else the same. Then note that for any $t\ge 2$ and any $x$,
\begin{align*}
|f(t,x) - \tf(t,x)|  &= \biggl|\max_{Q\in \cq_t} \sum_{i=0}^{t-1} F(z_{t-i, x+q_i}) - \max_{Q\in \cq_t} \sum_{i=0}^{t-2} F( z_{t-i, x+q_i})\biggr|\\
&\le  \max_{Q\in \cq_t}\biggl|\sum_{i=0}^{t-1} F(z_{t-i, x+q_i}) - \sum_{i=0}^{t-2} F(z_{t-i, x+q_i})\biggr|\\
&= \max_{Q\in \cq_t} |F(z_{1,x+q_{t-1}})|\le  \max_{y: |y|\le t} |F(z_{1,x+y})|. 
\end{align*}
Since the noise field is i.i.d.~Gaussian and $F$ is Lipschitz, it follows that 
\[
\ee\bigl( \max_{y: |y|\le t} |z_{1,x+y}|\bigr) \le C\sqrt{\log t}.
\]
Thus, the same upper bound holds for $\ee|f(t,x) - \tf(t,x)|$. Now note that $\tf(t,x)$ has the same law as $f(t-1,x)$.  This shows that
\begin{align*}
\ee(f(t,x)) - \ee(f(t-1,x)) &= \ee(f(t,x)) - \ee(\tf(t,x)) \\
&= \ee(f(t,x) - \tf(t,x)) \\
&\le \ee|f(t,x) - \tf(t,x)|\le C\sqrt{\log t}.
\end{align*}
By translation invariance, $\ee(f(t-1,x)) = \ee(f(t-1,y))$ for any $y$. Since the noise variables have mean zero, this shows that 
\begin{align*}
&\ee(f(t,x)) - \ee(f(t-1,x)) \\
&= \ee\biggl(f(t,x) - \frac{1}{d}\sum_{b\in B^+} f(t-1,x+b)\biggr)\\
&= \ee\biggl(\max_{b\in B^+} f(t-1,x+b) - \frac{1}{d}\sum_{b\in B^+} f(t-1,x+b)\biggr).
\end{align*}
By Lemma \ref{lppsimple}, 
\begin{align*}
&\max_{b\in B^+} f(t-1,x+b) - \frac{1}{d}\sum_{b\in B^+} f(t-1,x+b) \\
&\ge \frac{2}{d^3}\sum_{b,b'\in B^+, b\ne b'}|f(t-1,x+b) - f(t-1,x+b')|. 
\end{align*}
Combining this with the two preceding displays proves the first claim of the theorem. For the second, we combine the first claim with Theorem \ref{genvarthm} to get that for any $K$,
\begin{align*}
&\ee[(f(t,x)-f(t,y))^2] = \int_0^\infty 2r \pp(|f(t,x)-f(t,y)|\ge r) dr\\
&\le  \int_0^\infty C_1r \min\biggl\{\frac{\sqrt{\log t}}{r},  e^{- C_2r^2/t}\biggr\} dr\\
&\le \int_0^K C_1 \sqrt{\log t} dr + \int_K^\infty C_1 r e^{-C^2r^2/t} dr\\
&= C_1 K \sqrt{\log t} + C_3 t e^{-C_2 K^2/t}.
\end{align*}
Choosing $K$ to be a large enough multiple of $\sqrt{t \log t}$ completes the proof of the second claim of the theorem. The last claim now follows by Theorem \ref{equivthm2}.

\section{Proof of Theorem \ref{polymerthm}}\label{polymerproof}
Throughout this proof, $C, C_1, C_2, \ldots$ will denote constants that depend only on the dimension, the inverse temperature, and the law of the noise variables. Let $\cp_t$ be as in Subsection \ref{polymersec}. As in the proof of Theorem \ref{lppthm}, let us assume without loss of generality that the law of the environment has mean zero. We need  two simple lemmas. 
\begin{lmm}\label{easylmm}
For any $x\in \rr$, $\cosh x \ge e^{\min\{|x|, x^2\}/4}$.
\end{lmm}
\begin{proof}
First, suppose that $|x|\le 1$. Then note that
\begin{align*}
e^{x^2/4} &= 1 + \sum_{k=1}^\infty \frac{(x^2/4)^k}{k!}\le 1 + \sum_{k=1}^\infty (x^2/4)^k\\
&\le 1 + (x^2/4) \sum_{k=1}^\infty 4^{-(k-1)}= 1 + \frac{x^2}{3}\le \cosh x. 
\end{align*}
Next, consider $|x|>1$. Since $e \ge 1 + 1 +1/2 + 1/6 = 8/3$, we have $e^3 \ge (8/3)^3 = 512/27 > 16$, which gives $e^{3/4} > 2$. Thus, $e^{3|x|/4}> 2$, and hence $\cosh x \ge e^{|x|}/2\ge e^{|x|/4}$. 
\end{proof}
\begin{lmm}\label{polymersimple}
Let $x_1,\ldots,x_n$ be real numbers. Then 
\begin{align*}
&\log \biggl(\frac{1}{n}\sum_{i=1}^n e^{x_i}\biggr) - \frac{1}{n}\sum_{i=1}^n x_i \\
&\ge \frac{1}{4n^3} \min\biggl\{ \sum_{1\le i<j\le n} |x_i-x_j|,\sum_{1\le i<j\le n} (x_i-x_j)^2\biggr\}.
\end{align*}
\end{lmm}
\begin{proof}
Without loss of generality, suppose that $x_1$ is the largest and $x_2$ is the smallest among the $x_i$'s. By Lemma \ref{easylmm},
\begin{align*}
\frac{1}{2}(e^{x_1}+e^{x_2}) &= e^{(x_1+x_2)/2} \cosh ((x_1-x_2)/2)\\
&\ge e^{(x_1+x_2)/2 + \min\{|x_1-x_2|, (x_1-x_2)^2\}/16}.
\end{align*}
Thus, by Jensen's inequality,
\begin{align*}
&\frac{1}{n}\sum_{i=1}^n e^{x_i} = \frac{2}{n}\biggl(\frac{e^{x_1} + e^{x_2}}{2}\biggr) + \frac{1}{n}\sum_{i=3}^n e^{x_i}\\
&\ge \frac{2}{n} e^{(x_1+x_2)/2 + \min\{|x_1-x_2|, (x_1-x_2)^2\}/16} +   \frac{1}{n}\sum_{i=3}^n e^{x_i}\\
&\ge \exp \biggl(\frac{2}{n}((x_1+x_2)/2 + \min\{|x_1-x_2|, (x_1-x_2)^2\}/16) +   \frac{1}{n}\sum_{i=3}^n x_i\biggr)\\
&= \exp\biggl(\frac{1}{n}\sum_{i=1}^n x_i + \frac{1}{8n}\min\{|x_1-x_2|, (x_1-x_2)^2\}\biggr). 
\end{align*}
Taking logs on both sides and observing that $|x_1-x_2|\ge |x_i-x_j|$ for all $i$ and $j$ completes the proof.
\end{proof}

As in the proof of Theorem \ref{lppthm}, there exists $u^*\in \rr$ such that $F(u^*)=0$. Define a new surface $\tf$ by replacing $z_{1,x}$ by $u^*$ for all $x$, but keeping all else the same. Then note that for any $t\ge 2$ and any $x$,
\begin{align*}
f(t,x) - \tf(t,x) &= \frac{1}{\beta}\log \frac{\sum_{P\in \cp_t} \exp \bigl(\beta\sum_{i=0}^{t-1} F(z_{t-i, x+p_i})\bigr)}{\sum_{P\in \cp_t} \exp\bigl(\beta\sum_{i=0}^{t-2} F( z_{t-i, x+p_i})\bigr)}\\
&=  \frac{1}{\beta} \log \sum_{y\in \zz^d} \rho_{t-1,x}(y) \biggl(\frac{1}{2d}\sum_{b\in B} e^{\beta F(z_{1,y+b})} \biggr), 
\end{align*}
where 
\begin{align*}
\rho_{t-1,x}(y) &:= \frac{\sum_{P\in \cp_{t}, p_{t-2} = y}\exp \bigl(\beta\sum_{i=0}^{t-2} F(z_{t-i, x+p_i})\bigr) }{\sum_{P\in \cp_{t}}\exp \bigl(\beta\sum_{i=0}^{t-2} F(z_{t-i, x+p_i})\bigr)}.
\end{align*}
Let $\ee'$ denote taking expectation only in $\{z_{1,y}\}_{y\in \zz^d}$. Jensen's inequality gives
\begin{align*}
\ee'(f(t,x) - \tf(t,x)) &= \frac{1}{\beta} \ee'\biggl[\log \sum_{y\in \zz^d} \rho_{t-1,x}(y) \biggl(\frac{1}{2d}\sum_{b\in B} e^{\beta F(z_{1,y+b})} \biggr)\biggr]\\
&\le  \frac{1}{\beta}\log \sum_{y\in \zz^d} \ee'\biggl[\rho_{t-1,x}(y) \biggl(\frac{1}{2d}\sum_{b\in B} e^{\beta F(z_{1,y+b})} \biggr)\biggr]\\
&= \frac{1}{\beta} \log \sum_{y\in \zz^d} \rho_{t-1,x}(y) \ee'\biggl(\frac{1}{2d}\sum_{b\in B} e^{\beta F(z_{1,y+b})} \biggr).
\end{align*}
Since $\rho_{t-1,x}$ is a probability mass function on $\zz^d$, this shows that
\begin{align*}
\ee(f(t,x) - \tf(t,x)) &=  \ee[\ee'(f(t,x) - \tf(t,x))]\le C. 
\end{align*}
Note that $\tf(t,x)$ has the same law as $f(t-1,x) + \beta^{-1}\log (2d)$. 
Thus,
\begin{align*}
\ee(f(t,x)) - \ee(f(t-1,x) + \beta^{-1}\log (2d)) &= \ee(f(t,x)) - \ee(\tf(t,x)) \le C.
\end{align*}
By translation invariance, $\ee(f(t-1,x)) = \ee(f(t-1,y))$ for any $y$. Since the noise variables have mean zero, this shows that 
\begin{align*}
&\ee(f(t,x)) - \ee(f(t-1,x) + \beta^{-1}\log (2d)) \\
&= \ee\biggl(f(t,x) - \beta^{-1}\log (2d) - \frac{1}{2d}\sum_{b\in B} f(t-1,x+b)\biggr)\\
&= \ee\biggl[\frac{1}{\beta}\log \biggl(\frac{1}{2d}\sum_{b\in B} e^{\beta f(t-1,x+b)}\biggr) - \frac{1}{2d}\sum_{b\in B} f(t-1,x+b)\biggr].
\end{align*}
By Lemma \ref{polymersimple}, 
\begin{align*}
& \frac{1}{\beta}\log \biggl(\frac{1}{2d}\sum_{b\in B} e^{\beta f(t-1,x+b)}\biggr) - \frac{1}{2d}\sum_{b\in B} f(t-1,x+b)\\
&\ge \frac{1}{32d^3}\min\biggl\{\sum_{b,b'\in B} |f(t-1,x+b) - f(t-1, x+b')|, \\
&\qquad \qquad \sum_{b,b'\in B} (f(t-1,x+b) - f(t-1, x+b'))^2\biggr\}.
\end{align*}
Combining this with the two preceding displays, and noting that $|y-z|_1=2$ if and only if $y$ and $z$ are neighbors of some common vertex, we get that for any $y,z\in \zz^d$ with $|y-z|_1=2$, and any $t\ge 2$,
\begin{align*}
&\ee|f(t,y)-f(t,z)| \\
&\le 1 + \ee\min\{|f(t,y)-f(t,z)|, (f(t,y)-f(t,z))^2\} \le C.
\end{align*}
This proves the first claim of the theorem. The remaining claims can now be proved using similar tactics as in the proof of Theorem \ref{lppthm}.

\section*{Acknowledgments}
I thank Persi Diaconis for helpful comments that helped improve the first draft of the paper. I also thank the anonymous referee for a careful reading of the manuscript and pointing out several typos and errors.

\bibliographystyle{plainnat}
\bibliography{myrefs}

\end{document}